\documentclass[11pt,a4paper,twoside]{amsart}
\usepackage{amssymb}
\usepackage{mathtools}
\usepackage{graphics}
\usepackage{bm}
\usepackage{tikz}
\usepackage{xstring}
\usepackage{wasysym}
\usepackage{etoolbox}
\usepackage{enumerate}
\newtheorem{theo}{\bf Theorem}[section]
\newtheorem{prop}[theo]{\bf Proposition}
\newtheorem{lemma}[theo]{\bf Lemma}

\newtheorem{defi}[theo]{\bf Definition}
\newtheorem{coro}[theo]{\bf Corollary}

\theoremstyle{remark}

\newcommand{\mbf}{\mathbf}

\DeclareMathOperator{\wt}{wt}

\newcommand{\rei}{\varphi_{\begin{tinymatrix}  \cdot & \bullet \\ \bullet & \cdot \\ \end{tinymatrix}}}
\newcommand{\reiia}{\varphi_{\begin{tinymatrix}  \bullet & \bullet \\ \circ & \cdot \\ \end{tinymatrix}}}
\newcommand{\reiib}{\varphi_{\begin{tinymatrix}  \cdot &  \circ \\  \bullet & \bullet  \\ \end{tinymatrix}}}
\newcommand{\reiii} {\varphi_{\begin{tinymatrix}  \cdot & \bullet \\ \circ  & \cdot  \\ \end{tinymatrix}}}

\let\tinymatrix\smallmatrix

\patchcmd{\tinymatrix}{\scriptstyle}{\scriptscriptstyle}{}{}
\patchcmd{\tinymatrix}{\scriptstyle}{\scriptscriptstyle}{}{}
\patchcmd{\tinymatrix}{\vcenter}{\vtop}{}{}
\patchcmd{\tinymatrix}{\bgroup}{\bgroup\scriptsize}{}{}

\begin{document}
\title{TASEP in any Weyl Group}
\author{Erik Aas}
\address{Department of Mathematics, Royal Institute of Technology \\
  SE-100 44 Stockholm, Sweden}
\email{eaas@kth.se}
\date{April 2014}

\maketitle
{\bf Abstract.} We investigate a Markov chain defined by Thomas Lam \cite{Lam}, which generalizes the multi-type TASEP on a ring to any Weyl group. For groups of type C we define an analogue of the multiline queues of Ferrari and Martin (which compute the stationary distribution for the classical TASEP). While our construction does not suffice for finding the stationary distribution, the construction gives the stationary distribution of a certain projection of Lam's chain. Also, our approach is incremental, in the sense that the construction appears to fit into a pattern of 'conjugation matrices', which remains to be fully worked out. Finally, we prove a theorem for the classical TASEP which fits into the picture of viewing TASEP in a permutation-free way.

\section{Introduction}
\label{sec_intro}

Let $W$ be any finite Weyl group.
In \cite{Lam}, Lam defined an interesting Markov chain $\Theta'_W$ on $W$ whose stationary distribution is connected to the study of reduced expressions in the corresponding affine group.
It turns out \cite{AL} that for $W$ of type A, $\Theta'_W$ is the well-known multi-type TASEP on a ring, studied for some time by probabilists.
In particular its stationary distribution has a very elegant description in terms of the \emph{multi-line queues} of Ferrari and Martin \cite{FM}.
One could hope for a similar description of the stationary distribution of $\Theta'_W$ for general $W$.
This appears to be difficult.
Lam also defined a weighted variant $\Theta_W$ whose stationary distribution seems more amenable to analysis (though its connection to reduced words is unclear).

\begin{itemize}
	\item For each $W$, we identify certain projections of $\Theta_W$, which generalize the familiar operation of merging different particle classes in the TASEP. We state this in its most general form in Section 5.
	\item For $W$ of type C, we compute the stationary distribution of one of these projections. We also show how to 'invert' some of the projections. That is, we give a way to simulate one projection $\Theta_1$ in terms of a further projection $\Theta_2$ of it, together with independent randomness. One of these inversions can be seen as an analogue of the multi-line queues of Ferrari and Martin (see Theorem \ref{th_queue}).
	\item We give an equivalent description of the classical (type A) TASEP which seemingly can be formulated purely in terms of geometrical notions (i.e. independently of the permutation representation of the Weyl group of type A). Though it has some interest in itself, we do not manage to generalize it to other Weyl groups.
	\item In the final section we make some concluding remarks and pose several questions.
\end{itemize}

The paper is structured as follows. In Section \ref{sec_C}, we describe Lam's chain in the special case we are most interested in, the type C TASEP, and explain the notion of projections of Markov chains in this context. In Sections \ref{sec_fi} and \ref{sec_se} we prove two theorems partially describing the stationary distribution of $\Theta_W$ for $W$ of type $C$. In Section \ref{sec_general} we generalize some of the results in the preceding sections to arbitrary Weyl groups. In Section \ref{sec_ktasep} we prove a a theorem on parallel update rules for the classical TASEP on a ring, which coincides with $\Theta_W$ for $W$ of type A. Finally in Section \ref{sec_que} we pose some questions which we have no satisfactory answer to.

\section{Type C}
\label{sec_C}

Fix $n \geq 2$.
Throughout this section we fix $W$ to be the Weyl group of type C and rank $n$.
Rather than work through Lam's definition of $\Theta = \Theta_W$ (given in Section \ref{sec_general}), we give another definition which is easily checked to be equivalent (using the well-known permutation presentation of $W$, see \cite{BB}). 

A \emph{state} in $\Theta$ is an assignment of labeled circles to the sites of a cycle of length $2n$ drawn as in Figure \ref{fi_1} (in the case $n = 10$).

\begin{figure}
	\begin{tikzpicture}
	\node[draw,circle] at (1,0){$1$};
	\node at (1,1.5){$\cdot$};
	\node[draw,circle] at (2,1.5){$1$};
	\node at (2,0){$\cdot$};
	\node[draw,circle] at (3,1.5){$2$};
	\node at (3,0){$\cdot$};
	\node[draw,circle] at (4,1.5){$4$};
	\node at (4,0){$\cdot$};
	\node[draw,circle] at (5,0){$2$};
	\node at (5,1.5){$\cdot$};
	\node[draw,circle] at (6,0){$3$};
	\node at (6,1.5){$\cdot$};
	\node[draw,circle] at (7,1.5){$3$};
	\node at (7,0){$\cdot$};
	\node at (8,0){$\cdot$};
	\node at (8,1.5){$\cdot$};
	\node[draw,circle] at (9,0){$5$};
	\node at (9,1.5){$\cdot$};
	\node[draw,circle] at(10,1.5){$1$};
	\node at (10,0){$\cdot$};
	\draw(1.5,1.0)--(9.5,1.0);
 	\draw(1.5,0.5)--(9.5,0.5);
	\draw(1.5,1.0)--(1.5,0.5);
 	\draw(9.5,1.0)--(9.5,0.5);
	\end{tikzpicture}
	\caption{A state.}
	\label{fi_1}
\end{figure}

We refer to the circles in the diagram as \emph{particles}, and the numbers in them as their corresponding \emph{classes}. No two particles are allowed to occupy the same column, though columns are allowed to be empty (such as column number 8 from the left in the example).

There is an exponential bell with rate $1$ at each site. When a bell is activated, the particle at that position (if there is one) either (i) jumps to the next site $s$ counter-clockwise if that position is empty (ii) trades place with the particle at $s$ if that particle has higher class (iii) does nothing if the particle at $s$ has lower class. Moreover, the bell in the upper line in column $i$ and the bell in the lower line in column $i+1$ trigger each other, for each $1 \leq i < n$. This ensures that no two particles occupy the same column at any time. We may thus think of the pair of bells as a single bell with rate $2$. There are then $n$ bells with rate $2$ in the middle and two bells with rate $1$ in the leftmost and rightmost column.
We denote the action of a bell at site $i$ in the lower line and site $i+1$ in the upper line  by $\sigma_i$, for $1 \leq i < n$. The bell at the site furthest down to the right similarly defines $\sigma_n$ and the bell furthest up to the left defines $\sigma_0$.

This defines our Markov chain $\Theta$. Clearly the number $m_i$ of particles of each class $i$ is conserved by the dynamics. We refer to the vector $\mbf{m} = (m_1, m_2, \dots)$ as the {\it type} of the state.

When there is one particle of each class, and each column is occupied by one particle, states are in one-to-one correspondence with signed permutations $w$ of $[n]$ -- there is a particle of class $i$ in the upper row in column $j$ if $w(j) = +i$ and in the lower row if $w(j) = -i$. In the general case, states correspond to left cosets of the group of signed permutations -- we will explain this in more detail later.

For $J \subseteq [n]$ we define a type $\mbf{m}_J$ as follows. Start with a state of type $(1,1,\dots,1)$. If $j \in J - \{n\}$, identify the particle classes $j$ and $j+1$. If $n \in J$, remove all particles of class $n$. Now renumber the remaining particle classes with integers $1, 2, \dots$. The resulting state has a type which we define to be $\mbf{m}_J$. For example, the states in Figures \ref{fi_1} and \ref{fi_2} have type $\mbf{m}_{\{1, 2,4,6,10\}}$ and $\mbf{m}_{\{1,2,4,6\}}$ respectively. Denote the set of all states of type $\mbf{m}_J$ by $\Omega_J$.

Clearly the $\mbf{m}_J$ enumerate all the interesting variations of types of states. Using the identifications above we similarly obtain a projection map $\varphi_i$ from words of type $\mbf{m}_J$ to words of type $\mbf{m}_{J\cup\{i\}}$ for each $J$ such that $i\notin J$.
We denote the restriction of $\Theta$ to states of type $\mbf{m}_J$ by $\Theta_J$. It is easy to see that $\Theta_J$ is aperiodic and irreducible. Thus it has a unique stationary distribution $\pi_J$.

Suppose $(i, J, J')$ satisfies $i \notin J\subseteq [n]$, $J' = J \cup \{i\}$. We should think of $(i, J, J')$ as a covering relation in the boolean poset of all subsets of $[n]$, and call such a triple a {\it link}. Of course any two of $i,J$ and $J'$ determine the third, but the terminology of links turns out to be convenient.

\begin{prop}
	\label{proj1}
	Suppose $(X_1, X_2, \dots)$ is a random walk in $\Theta_J$, and $(i,J,J')$ is a link. Then $(\varphi_i(X_1), \varphi_i(X_2), \dots)$ is a random walk in $\Theta_{J'}$.
\end{prop}

\begin{defi}
	For a Markov chain $\Theta$ with state space $\Omega$, we define its $\Omega\times\Omega$ {\it transition matrix} $M_\Theta$ by letting 
	\[
		M_\Theta(u, v) = \sum_{\gamma : u \to v}  \textrm{rate}(\gamma),
	\]
	the sum of the rates of all transitions $\gamma$ from state $u$ to state $v$.
\end{defi}

Write $M_J$ for $M_{\Theta_J}$. Thus the stationary distribution $\pi_J$ is an eigenvector of $M_J$, with eigenvalue $2n$.

Suppose $(i, J, J')$ is a link, and define a $\Omega_{J'}\times \Omega_J$ matrix $D = D_{i,J}$ by letting $D(v, u)$ be $1$ if $v = \varphi_i(u)$, and $0$ otherwise. Proposition \ref{proj1} can be strengthened as follows.

\begin{prop}
	\label{proj2}
	Suppose $(i,J,J')$ is a link. Then
	\[
		D_{i,J}M_J = M_{J'}D_{i,J}.
	\]
\end{prop}
\begin{proof}
	Choose $(v, u) \in \Omega_{J'} \times \Omega_J$.
	The $(v,u)$ entry of the left hand side counts the number of $j\in [0, n]$ (where $0$ and $n$ are counted with weight $1$ and the others with weight $2$) such that $v = \varphi_i(\sigma_j u)$.
	Similarly, the $(v,u)$ entry of the right hand side counts the number of $j'$ such that $v = \sigma_{j'} (\varphi_i(u))$.
	These two counts can be matched to each other simply by taking $j = j'$.
\end{proof}

Proposition \ref{proj2} is stronger than Proposition \ref{proj1} in the sense that for \emph{any} eigenvector $v$ of $M_J$ with eigenvalue $\lambda$, we have $\lambda D_{i,J}v = D_{i,J}M_Jv = M_{J'} (D_{i,J}v)$, so that $D_{i,J}v$ is an eigenvector of $M_{J'}$ if it is non-zero.
Since $\pi_J$ spans a one-dimensional eigenspace of $M_J$ for each $J$, $D_{i,J}\pi_J$ must be a scalar multiple of $\pi_{J'}$. Though we are only interested in the stationary distribution of the $\Theta_J$, considering \emph{all} the eigenvectors will help us to compute the stationary distribution $\pi_J$. \footnote{In \cite{AAMP}, the authors investigate the eigenvalues and eigenvectors for the TASEP on a ring, ie. our chain when $W$ is of type A.} 

Of course, Proposition \ref{proj2} appears to have no practical use -- the chain $\Theta_J$ is more complicated than the chain $\Theta_{J'}$. Therefore, the following fact from linear algebra is quite a revelation.

\begin{prop}
	\label{linalg}
	Suppose $A, B$ are matrices. If there is a matrix $D$ of full rank such that 
	\[
		DA = BD,
	\]
	then there is a matrix $U$ of full rank such that
	\[
		AU = UB.
	\]
\end{prop}

So the mere existence of the {\it projection matrix} $D_{i,J}$ implies the existence of some {\it conjugation matrix} $U_{i,J}$ such that 
\[
	M_J U_{i,J} = U_{i,J}M_{J'},
\]
which would allow us to compute $\pi_J$ from $\pi_{J'}$ by $\pi_{J} = U_{i,J}\pi_{J'}$! Of course there is no guarantee that there will be a "simple" matrix $U_{i,J}$ (eg. with small positive integer entries). However, in the coming sections we will identify such $U_{i,J}$ for some links $(i,J,J')$. It would be very interesting if these $U_{i,J}$ could be defined at the generality of Proposition \ref{linalg} -- i.e. as a function of $(M_J,M_{J'},D_{i,J})$. The two cases we consider will correspond to adding/removing the particles of largest class, and the case of only two classes of particles.

\section{Adding particles of largest class}
\label{sec_fi}

In this section we construct a conjugation matrix $U = U_{i,J}$ for links $(i, J, J')$ with $i = n$ and $J\subseteq[n-1]$ arbitrary. To describe $U$, take a state $u$ of type $\mbf{m}_{J'}$, and let $\vartheta \in \{+, -\}^n$. We will produce a new state $v = \tau_\vartheta u$, of type $\mbf{m}_J$. Column $j$ in $v$ will be occupied in the upper line if $\vartheta_j = +$ and in the lower line otherwise (in particular, there will be no empty columns in $v$). In the start we refer to these sites as Not Yet Occupied.
Go through the particles in $u$ in any order such that particles with small class come before particles with larger class. The order of particles of same class does not matter. Now, when considering a particle $p$ at a site $s$ in $u$, find the first NYO site in $v$, going \emph{clockwise} from $s$. Put $p$ at this site. Now $s$ has been occupied. When all particles in $u$ have been processed, fill the remaining NYO sites in $v$ with particles of a new largest class.

In Figure \ref{fi_2}, we have applied $\tau_\vartheta$, where $\vartheta=(+,-,-,+,+,-,-,-,-,-)$, to the state in Figure \ref{fi_1}.
We define $U$ by letting its $(u,v)$ entry be the number of $\vartheta\in\{+,-\}^n$ such that $v = \tau_\vartheta u$.

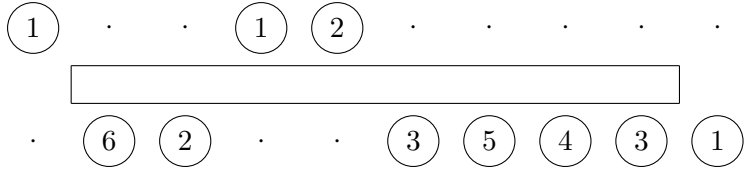
\begin{figure}
	\begin{tikzpicture}
		\node[draw,circle] at (1,1.5){$1$};
		\node at (1,0){$\cdot$};
		\node[draw,circle] at (2,0){$6$};
		\node at (2,1.5){$\cdot$};
		\node[draw,circle] at (3,0){$2$};
		\node at (3,1.5){$\cdot$};
		\node[draw,circle] at (4,1.5){$1$};
		\node at (4,0){$\cdot$};
		\node[draw,circle] at (5,1.5){$2$};
		\node at (5,0){$\cdot$};
		\node[draw,circle] at (6,0){$3$};
		\node at (6,1.5){$\cdot$};
		\node[draw,circle] at (7,0){$5$};
		\node at (7,1.5){$\cdot$};
		\node[draw,circle] at (8,0){$4$};
		\node at (8,1.5){$\cdot$};
		\node[draw,circle] at (9,0){$3$};
		\node at (9,1.5){$\cdot$};
		\node[draw,circle] at(10,0){$1$};
		\node at (10,1.5){$\cdot$};
		\draw(1.5,1.0)--(9.5,1.0);
  	\draw(1.5,0.5)--(9.5,0.5);
		\draw(1.5,1.0)--(1.5,0.5);
  	\draw(9.5,1.0)--(9.5,0.5);
	\end{tikzpicture}
	\caption{The result of applying $\tau_{(+,-,-,+,+,-,-,-,-,-)}$ to the state in Figure \ref{fi_1}.}
	\label{fi_2}
\end{figure}

\begin{theo}
\label{th_queue}
	In the notation above, we have 
	\[
		U M_{J'} = M_J U.
	\]
\end{theo}
\begin{proof}
	Fix $u \in \Omega_{J'}$. We need to show that the number (counted with the same weights as before) of $(j, \vartheta)$ such that $v = \sigma_j \tau_\vartheta u$ equals the number of $(j', \vartheta')$ such that $v = \tau_{\vartheta'} \sigma_{j'} u$, for each $v \in \Omega_J$.

Some pairs $(j, \vartheta)$ have a natural corresponding pair $(j',\vartheta')$ with $j' = j$, with same weight and such that $\tau_\vartheta \sigma_j u = \sigma_{j'}\tau_{\vartheta'}$, as follows.

\begin{itemize}
	\item if $0 < j < n$ and $\vartheta_j = \vartheta_{j+1}$, let $\vartheta' = \vartheta$.
	\item if $0 < j < n$ and $(\vartheta_{j},\vartheta_{j+1}) = (-, +)$, let $\vartheta' = (\vartheta_1, \dots, \vartheta_{j-1}, \vartheta_{j+1}, \vartheta_{j}, \vartheta_{j+2}, \dots, \vartheta_n)$.
	\item if $j = 0$ and $\vartheta_1 = -$, let $\vartheta'= (+,\vartheta_2,\dots,\vartheta_n)$.
	\item if $j = n$ and $\vartheta_n = +$, let $\vartheta'=(\vartheta_1,\dots,\vartheta_{n-1},-)$.
\end{itemize}

We illustrate the first of these in Figure \ref{fi_3}.
\begin{figure}
	\begin{tikzpicture}
		\node[draw,circle] at (1,0){$?$};
		\node[draw,circle] at (2,0){$?$};
		\node[draw,circle] at (3,1.5){$1$};
		\node[draw,circle] at (4,1.5){$3$};
		\node at (0,0){$\dots$};
		\node at (0,1.5){$\dots$};
		\node at (5,0){$\dots$};
		\node at (5,1.5){$\dots$};
		\node at (2.5,2){$\leftarrow$}; 
		\draw(0.5,1.0)--(4.5,1.0);
  	\draw(0.5,0.5)--(4.5,0.5);

		\node at (6.5,1){$\sigma_{j}$};
		\node at (6.5,0.5){$\longrightarrow$};
	
		\node[draw,circle] at (9,0){$?$};
		\node[draw,circle] at (10,1.5){$1$};
		\node[draw,circle] at (11,0){$?$};
		\node[draw,circle] at (12,1.5){$3$};
		\node at (8,0){$\dots$};
		\node at (8,1.5){$\dots$};
		\node at (13,0){$\dots$};
		\node at (13,1.5){$\dots$};
		\draw(8.5,1.0)--(12.5,1.0);
  	\draw(8.5,0.5)--(12.5,0.5);
	
		\node at (2.7,3.5){$\tau_{\vartheta}$};
		\node at (3,3.5){$\downarrow$};
	
		\node[draw,circle] at (1,6.5){$1$};
		\node[draw,circle] at (2,6.5){$5$};
		\node[draw,circle] at (3,6.5){$3$};
		\node[draw,circle] at (4,5){$2$};
		\node at (0,5){$\dots$};
		\node at (0,6.5){$\dots$};
		\node at (5,5){$\dots$};
		\node at (5,6.5){$\dots$};
		\draw(0.5,6.0)--(4.5,6.0);
  	\draw(0.5,5.5)--(4.5,5.5);
		\node at (1,5.75){$\downarrow$};
		\node at (2,5.75){$\downarrow$};
		\node at (3,5.75){$\uparrow$};
		\node at (4,5.75){$\uparrow$};
		\node at (2.5,7){$\leftarrow$};

		\node at (6.5,6){$\sigma_{j'}$};
		\node at (6.5,5.5){$\longrightarrow$};
	
		\node[draw,circle] at (9,6.5){$1$};
		\node[draw,circle] at (10,6.5){$3$};
		\node[draw,circle] at (11,6.5){$5$};
		\node[draw,circle] at (12,5){$2$};
		\node at (8,5){$\dots$};
		\node at (8,6.5){$\dots$};
		\node at (13,5){$\dots$};
		\node at (13,6.5){$\dots$};
		\draw(8.5,6.0)--(12.5,6.0);
  	\draw(8.5,5.5)--(12.5,5.5);
		\node at (9,5.75){$\downarrow$};
		\node at (10,5.75){$\uparrow$};
		\node at (11,5.75){$\downarrow$};
		\node at (12,5.75){$\uparrow$};
	
		\node at (10.7,3.5){$\tau_{\vartheta'}$};
		\node at   (11,3.5){$\downarrow$};

\end{tikzpicture}
	\caption{In the top left corner we have drawn a state $u$, and indicated the actions of $\sigma_{j'}$ and $\tau_{\vartheta}$.
Below and to the right are $\sigma_{j'}u$ and $\tau_{\vartheta}u$ with $\tau_{\vartheta'}$ and $\sigma_j$ indicated respectively.
The claim is that the state in the bottom right is both $\sigma_j \tau_{\vartheta} u$ and $\tau_{\vartheta'}\sigma_{j'}u$.
We have only indicated one possible result $\tau_{\vartheta}u$; it could be that the position labeled $3$ is replaced by a particle with smaller class.
However, since there are no particles of class $< 1$ we are certain that the position labeled $1$ is labeled so in any circumstance.
It is easy to argue for each particle that it is going to end up in the same place in both $\tau_{\vartheta'}\sigma_{j'}u$ and $\sigma_j\tau_{\vartheta}u$.}
	\label{fi_3}
\end{figure}
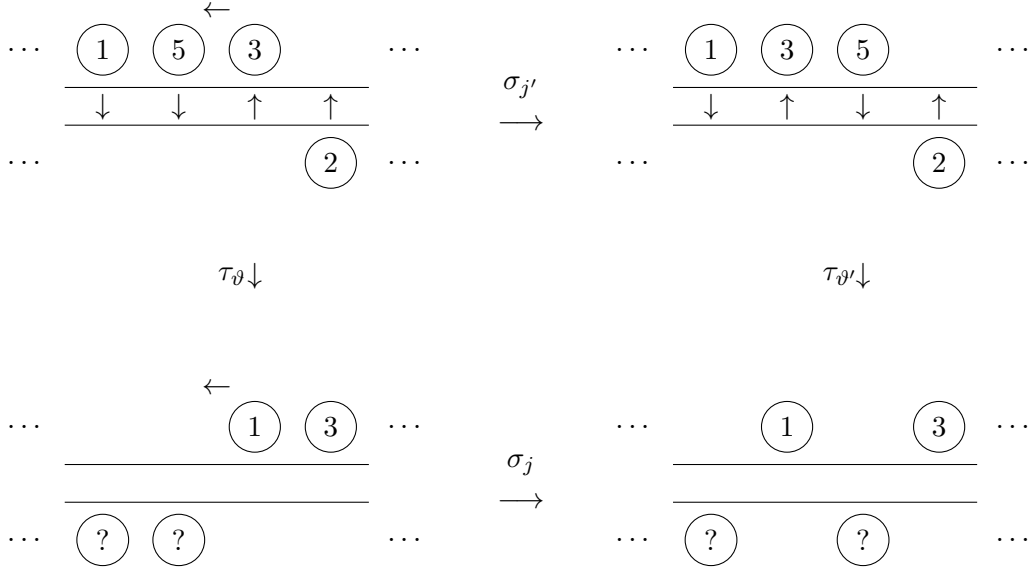

Thus we need to show that the set $S$ of remaining pairs $(j, \vartheta)$, ie. those satisfying
\begin{itemize}
	\item $0 < j < n$, $(\vartheta_{j}, \vartheta_{j+1})=(-,+)$,
	\item $j = 0$, $\vartheta_1 = +$, or
	\item $j = n$, $\vartheta_n = -$,
\end{itemize}

have the same effect as the set $S'$ of remaining pairs $(j', \vartheta')$, ie. those satisfying

\begin{itemize}
	\item $0 < j < n$, $(\vartheta_{j}, \vartheta_{j+1}) = (+,-)$,
	\item $j = 0$, $\vartheta'_1 = -$, or
	\item $j = n$, $\vartheta'_n = +$.
\end{itemize}

It is easy to see that for all $(j, \vartheta) \in S$, the state $\tau_\vartheta \sigma_j u$ is the same, and equal to $\sigma_{j'}\tau_{\vartheta'} u$ for all $(j', \vartheta')\in S'$.
Thus it suffices to show that their (weighted) counts are the same! This is easily done by considering separately the cases $(\vartheta_1, \vartheta_n) = (+,+), (+,-), (-,+), (-,-)$ and similarly for $(\vartheta'_1, \vartheta'_n)$.

\end{proof}

Thus, if $u$ is distributed according to $\pi_{J'}$ and $\vartheta \in \{+,-\}^n$ is chosen indepedently and uniformly at random, then $\tau_\vartheta u$ will be distributed according to $\pi_J$.

\begin{coro}
	The Markov chain on $\Omega_J$ obtained by at each time step applying a random $\tau_\vartheta$ has the same stationary distribution as $\Theta_J$.
\end{coro}
\begin{proof}
	In the proof of Theorem \ref{th_queue}, we never used the fact that there was some empty column in the state $u$. Thus the same proof shows that $M_J U = U M_J$ where $U(u, v) = 1$ if $v = \tau_\vartheta$ for some $\vartheta$ and $0$ otherwise, for $(u,v)\in\Omega_J\times\Omega_J$. Thus $U$ maps $\pi_J$ onto a constant multiple of itself.
\end{proof}

\section{Two particle classes}
\label{sec_se}

In this section we construct $U = U_{t, J}$ for the case $J = [n]-\{t\}$, $t \neq n$ (note that the case $t = n$ is trivial). In this case $U$ will only have one column, so this is equivalent to describing the stationary state of $\Theta_J$. We will not phrase the result in terms of $U$.

The chain $\Theta_J$ has strong similarities with the 3-type TASEP on a ring, and with the chain studied in \cite{DEHP}. Indeed we will show that the stationary distribution satisfies similar recursion relations to these two chains (only the initial data will be different), and the proof is very similar to that of \cite{DEHP}.

When $J = [n]- \{t\}$, we are considering states with $t$ particles of class $1$. It will be more convenient to write states as words, as follows. A state $u$ is described by a word $w_1, \dots, w_n \in \{-1,0,1\}^n$, where $w_i = 1$ if there is a particle (with the only class $1$) in the upper line in column $i$, $w_i = -1$ if there is one in the lower line, and $w_i = 0$ if both lines are empty in this column. We sometimes write $-x = \bar{x}$ for $x = 1, 0, \bar{1}$. It would be more natural to write $1,2,3$ for $1,0,\bar{1}$, though we will not do this in order to be consistent with previous sections.

In this notation, $\Theta_J$ becomes a chain on words, where any subword $01,\bar{1}1,\bar{1}0$ turns into $10,1\bar{1},0\bar{1}$ respectively at rate $2$, a $\bar{1}$ on the right end turns into $1$ at rate $1$, and a $1$ at the left end turns into a $\bar{1}$ at rate $1$.

Our analysis becomes more transparent if we temporarily generalize $\Theta_J$ so that a subword $01$ turns into $10$ at rate $a$, $\bar{1}1$ into $1\bar{1}$ at rate $b$, $\bar{1}0$ into $0\bar{1}$ at rate $c$, $\bar{1}$ on the right end into $1$ at rate $d$, and $1$ at the left end into $\bar{1}$ at rate $e$, for indeterminates $a,b,c,d,e$.

Furthermore, we will not consider the length $n$ of the words to be fixed anymore; we want to describe the value of stationary distribution at \emph{any} word in $\{-1,0,1\}$ of length at least $2$.

Thus our case of interest is $(a,b,c,d,e)$ proportional to $(2,2,2,1,1)$. To get conventions right, the most convenient choice will be $(a,b,c,d,e) = (1,1,1,\frac{1}{2},\frac{1}{2})$.

\begin{defi}
For each word $u$ (of any length $\geq 2$) we define a Laurent polynomial $[u]$ in $a,b,c,d,e$ as follows.	Let $v,w$ be any words (possibly empty). Then 
								\begin{equation}\label{recrel1}
		[v01w] = [v0w]/a,
	\end{equation}\begin{equation}\label{recrel2}
		[v\bar{1}1w] = ([v\bar{1}w] + [v1w])/b,
	\end{equation}\begin{equation}\label{recrel3}
		[v\bar{1}0w] = [v0w]/c,
	\end{equation}\begin{equation}\label{recrel4}
		[v\bar{1}] = [v] / d, 
	\end{equation}\begin{equation}\label{recrel5}
		[1v] = [v] / e.
	\end{equation}
	Moreover, if $u$ consists of $0$s only, then $[u] = 1$.
\end{defi}

To prove that this defines $[u]$ in a unique way, we need to show that when expanding according to the recursions above we always arrive at the same result $t \cdot [0^s]$, where $t$ is some Laurent polynomial and $s$ is the number of $0$'s in $u$ (which is clearly conserved by the recursions - in fact we could have made any set of choices for $[], [0], [00], [000], \dots$ rather than setting them all equal to $1$). This is easy by induction, using the following induction hypothesis:

\emph{For words of length $\leq n$, taking any definition of $[u]$ for each $u$, all equations between brackets of words of length $\leq n$ are satisfied.}

Here is an interesting consequence of the definition.

\begin{prop}
	If $u = v_1 0 v_2 0 \dots 0 v_r$ where $v_1,\dots,v_r$ are any words in $\{-1,0,1\}$, then $[u] = [v_1 0][0 v_2 0]\dots [0 v_r]$.
\end{prop}

\begin{theo}
	For any $n$ and $t$, the stationary distribution of $\Theta_J$, where $J = [n] - \{t\}$, evaluated at the state $u$ is proportional to $[u]$.
\end{theo}

\begin{proof}
It suffices to prove that the numbers $[u]$ satisfy the equilibrium equation, which reads
\begin{equation}
	\label{eqeq}
		0 = T_0 + \sum_{0 < i < n} T_i + T_n, 
\end{equation}
	where
\[
	T_0 = \wt(u_1\to \bar{u_1})[u_1\dots u_n] - \wt(\bar{u_1}\to u_1)[\bar{u_1} u_2 \dots u_n],
\]

\[
	T_i = \wt(u_iu_{i+1} \to u_{i+1} u_i)[u_1\dots u_n] - \wt(u_{i+1}u_i\to u_iu_{i+1}) [u_1\dots u_{i-1}u_{i+1}u_iu_{i+2}\dots u_n],
\]
for $0 < i <n$ and
\[
	T_n = \wt(u_n \to \bar{u_n})[u_1\dots u_n] - \wt(\bar{u_n} \to u_n)[u_1\dots u_{n-1} \bar{u_n}].
\]

Here $\wt(01 \to 10) = a$, $\wt(1\bar{1} \to \bar{1}1) = 0$, $\wt(1 \to \bar{1}) = e$ at the left end, et.c.
A case by case analysis (using \ref{recrel1}-\ref{recrel5}) shows that
\[
	T_0 = a_{u_1} [u_2 \dots u_n],
\]
\[
	T_i = a_{u_i} [u_1\dots u_{i-1} u_{i+1} \dots u_n] - a_{u_{i+1}}[u_1\dots u_i u_{i+2} \dots u_n]
\]
	for $0 < i < n$ and  
\[
	T_n = -a_{u_n} [u_1\dots u_{n-1}]
\]

where $(a_1, a_0, a_{\bar{1}}) = (1, 0, -1)$.

This turns the right hand side of (\ref{eqeq}) into a telescoping sum with value $0$ (each term $a_{u_i}[\dots \hat{u_i} \dots]$ occurs twice - once with each sign), so the equilibrium equation is satisfied. 

\end{proof}

For the case we are interested in, $(a,b,c,d,e)=(1,1,1,\frac{1}{2},\frac{1}{2})$, we obtain

\begin{coro}

For each $n, t$, choose $\alpha_{n,t} > 0$ such that the minimum of the numbers $n_u = \alpha_{n,t} \pi_J(u)$ for $u \in \Omega_{[n]-t}$ is $1$. Then

\begin{itemize}
	\item Each $n_u$ is a positive integer.
	\item $n_u = 1$ if and only if $u = \bar{1}^i0^j1^k$ for some $i,j,k$ (satisfying $i+k=t$, $i+j+k = n$).
	\item $n_u \leq 2^t$ for all $u$, with equality if and only if $u = 1^i 0^j \bar{1}^k$ for some $i,j,k$.
	\item For any words $u, v, w$ in $\{-1,0,1\}$, $n_{u0v0w} = n_{u0} \cdot n_{0v0} \cdot n_{0w}$.
\end{itemize}
\end{coro}

\section{The general case}
\label{sec_general}

In this section, we explain how much of the analysis carries over to general Weyl groups. The short answer is that the general setup can be formulated for a general Weyl group, but the lack of a useful permutation representation makes it hard to describe the conjugation matrices in an easy way.

Fix a root system $\Phi$ of rank $n$ in $\mathbb{R}^N$ (for some $N$) with a simple system $\alpha_1, \dots, \alpha_n$. Let $\alpha_0$ be the highest root with respect to this choice of simple system, and write $a_0 \alpha_0 = \sum_{i=1}^n a_i \alpha_i$ where the $a_i$ are nonnegative integers with $a_0 = 1$. For $i \in [0, n]$, we let $t_i$ be the reflection in the hyperplane orthogonal to $\alpha_i$. Then $t_1,\dots,t_n$ generate the Weyl group of $\Phi$. We denote its length function by $\ell(\cdot)$. For an element $w \in W$ and $i \neq 0$, we define $\sigma_i(w)$ to be equal to $w$ if $\ell(wt_i) > \ell(w)$ and equal to $wt_i$ otherwise. For $i = 0$, we let $\sigma_i(w) = w$ if $\ell(wt_0) < \ell(w)$ and $wt_i$ otherwise. 

The following is Lam's original definition of $\Theta$.
\begin{defi}
	The Markov chain $\Theta$ has $W$ as state space, and the outgoing transition from any state $w \in W$ are all $w \to \sigma_i(w)$, $i \in [0, n]$. Transitions corresponding to $\sigma_i$ have rate $a_i$.
\end{defi}

For type A, we have $(a_0, \dots, a_n) = (1, \dots, 1)$ and the chain $\Theta'$ in Section \ref{sec_intro} coincides with $\Theta$ in this case. For type C we have $(a_0, a_1, \dots, a_{n-1}, a_n) = (1,2, \dots, 2, 1)$. The chain for type C is equivalent to the chain defined in Section \ref{sec_intro}. 

It appears that the above choice of rates $(a_0, \dots, a_n)$ is essentially the only one that gives a nice stationary distribution for $\Theta$ (eg. such that the probability of any state divided by the probability of a least likely state is an integer, and these integers are not 'too large'). I have no good explanation for this experimental fact; however, see Remark 5 in \cite{Lam}.

We now show how to generalize Propositions \ref{proj1} and \ref{proj2} to the general setting. For $J \subseteq [n]$, let $W_J$ be the subgroup of $W$ generated by $\{t_j : j \in J\}$. Note that we do not allow $0 \in J$.

\begin{prop}
	Suppose $w, w'\in W$ satisfy $W_J w = W_J w'$. Then, for any $i\in[0,n]$, we have $W_J (\sigma_i(w)) = W_J(\sigma_i(w'))$.
\end{prop}

\begin{proof}
This is clear if $\ell(wt_i) > \ell(w)$ and $\ell(w't_i) > \ell(w')$, or if $\ell(wt_i) < \ell(w)$ and $\ell(w't_i) < \ell(w')$.

We can thus focus on the case when $\ell(wt_i) > \ell(w)$ and $\ell(w't_i) < \ell(w')$, the fourth case then follows by symmetry in $w, w'$.
It suffices to prove that $wt_iw^{-1} \in W_J$ since this will prove that $W_J \sigma_iw' = W_J w't_i = W_J w t_i = W_J w = W_J \sigma_i w$ if $i \neq 0$ and	$W_J \sigma_iw' = W_J w' = W_J w = W_J wt_i = W_J \sigma_i w$ if $i = 0$.
Since each $W_J$ has a minimal element (Corollary 2.4.5 in \cite{BB}), we can reduce to the case when $w < w'$ in Bruhat order (the case $w' < w$ is similar).

Then (see the mirrored version of Proposition 2.4.4 in \cite{BB}), $w$ has a reduced expression
\[
	s_1\dots s_r,
\]
and	$w'$ has a reduced expression
\[
	s'_1 \dots s'_l
\]
where $s'_i \in W_J$.

The conditions imply (Corollary 1.4.4 in \cite{BB}) that $t_i=s_r\dots s_1 s'_l \dots s'_k \dots s'_l s_1 \dots s_r$ for some $k$ and thus $ut_iu^{-1}  = s'_l\dots s'_k\dots s'_l \in W_J$.
\end{proof}
Thus all the subgroup inclusion $W_{J'} \leq W_J$ (with $J' \subseteq J$) induce projections $D_{J,J'}:\Theta_J \to \Theta_{J'}$ in the same way as in Section \ref{sec_C}.
Explicitly, let $D_{i, J}(W_Ju, W_{J'}v) = 1$ if $W_{J'}u = W_{J'}v$ and $0$ otherwise. Then $D_{i,J}M_J = M_{J'}D_{i,J}$.
We should thus be able to expect there to be corresponding conjugation matrices $U_{J,J'}$. We have not succeeded in finding a general description of these (indeed, not even for groups of type C), though there appears to be 'surprisingly simple' such $U$ in several cases.

\section{The $k$-TASEP}
\label{sec_ktasep}

In this section, we prove a theorem about the TASEP on a ring (i.e. the type A case of the chains considered above), which can be formulated independently of the permutation representation. However the most obvious generalizations of the theorem do not appear to hold for a general Weyl group. 

We now define the classical multi-type TASEP on a ring. We will include indeteminates $x_1, x_2, \dots$ in the definition as in \cite{LW}. This is more general than the chain $\Theta_W$ considered earlier, but reduces to it on letting $x_1 = x_2 = \dots = 1$.

The state space is the set of words of length $n$ in the alphabet $\{1, 2, \dots\}$. We consider the words to be cyclic, so that the letter to the left of $w_1$ is $w_0 = w_n$ (all indices are taken modulo $n$). For $i\in[n]$ we define $\sigma_i(w)$ as the result of sorting the two letters $w_i$ and $w_{i-1}$. Thus if the letters already satisfy $w_{i-1} \leq w_i$, nothing happens, and otherwise they swap positions.

The outgoing transitions from a general state (word) $u$ are all $u \to \sigma_i u$, $i\in[n]$, where transitions corresponding to $\sigma_i$ have rate $x_{u_i}$ (note that this definition is \emph{not} symmetric in $(u_{i-1},u_i)$).

For a proper subset $S \subsetneq [n]$ we would like to define $\sigma_S$ as the composition of all $\sigma_j$, $j\in S$. To do this we need to specify in which order non-commuting pairs of $\sigma_j$'s should be taken. Note that $\sigma_j$ and $\sigma_k$ commute whenever $|k-j|>1$ modulo $n$. We use the convention that \emph{$\sigma_{j-1}$ is taken before $\sigma_j$}.
Thus, for example, if $n = 7$, $\sigma_{ \{1,2,4,5,7\}}$ equals $\sigma_2\sigma_1\sigma_7\sigma_5\sigma_4$.
The transition $u \to \sigma_S u$ is given rate $\prod_{i\in S} x_{u_i}$.

\begin{theo}
	\label{ktasep}
	Fix $k \in (0,n)$. The Markov chain $\Theta_k$ on words whose outgoing transitions from a word $u$ are all $u \to \sigma_Su$, $S\subseteq[n]$, $|S|=k$, has the same stationary distribution for all $k$.
\end{theo}

In particular $\Theta_k$ has the same stationary distribution as $\Theta_1$ for each $k$, where $\Theta_1$ is the inhomogenous TASEP on a ring introduced in \cite{LW}. 

A weaker theorem has been proved by Martin and Schmidt \cite{MS}, where the underlying graph is the infinite discrete line, all $x_i = 1$, and subsets $S$ of size $k$ are instead chosen at rate $p^k$, where $p \in (0,1)$ is a parameter. Since our proof will be purely local, the theorem above implies their result.

The theorem can be reduced to proving that the transition matrices of all the $\Theta_k$ commute among themselves. To see this, note that the stationary distribution corresponds to the largest eigenvalue $n$ and that the corresponding eigenspace is one-dimensional. In the next section we prove this commutation property.

Before that, let us note that in the homogenous case, $x_1 = \dots = 1$, the chain $M_k$ in Theorem \ref{ktasep} has a purely geometric definition: randomly compose a subset $S$ of $k$ simple generators (together with the reflection in the highest root), ordering them according to some orientation of the affine Dynkin diagram (the graph with one node for each generator and one node for the highest root, where two nodes are joined by an edge if their order is $>2$ (i.e. if they do not commute)). We have tried all possible orientations of the affine Dynkin graph (on $4$ nodes) of the group $B_3$, but none is consistent with a result analogous to Theorem \ref{ktasep}.

Finally, we remark that the operators $\sigma_S$ are closely related to the multi-line queues \cite{FM}. Indeed, a special case of the case $k = n-1$ of the theorem follows directly from the theory of multi-line queues: for readers familiar with multi-line queues, this is given by comparing the $(n-1)$-TASEP with multiline queues which have an extra last row to which no new particles are added. Thus, Theorem \ref{ktasep} interpolates between the definition of the chain ($k = 1$) and its highly non-trivial description in terms of multi-line queues ($k=n-1$).

\subsection{Proof of Theorem \ref{ktasep}}

As noted earlier, it suffices to show that for any $k, l \in (0, n)$, the transition matrices $A_k$, $A_l$ of $M_k$ and $M_l$ commute. 

For words $u, v$, the $(u,v)$ entry of $A_kA_l$ is a weighted count of pairs $(S, T)$ such that $v = \sigma_T \sigma_S u$, $|S| = l$ and $|T| = k$ (weighted by the product of the rates of the two transitions $u \to \sigma_S u \to \sigma_T \sigma_S u$). Similarly, the entry $(A_lA_k)(u,v)$ counts pairs $(S', T')$ such that $ v = \sigma_{T'}\sigma_{S'} u$, $|S'| = k$, $|T'|=l$. The idea of the proof is to find a weight-preserving involution from pairs $(S,T)$ to $(S', T')$. This turns out to be tricky. Without loss of generality, we will only consider words with distinct letters -- the statement for general words follows by merging particle classes.

Given a triple $(u, S, T)$, define a $2\times n$ array $D$ (cyclic in the horizontal direction) as follows. If $i \in S$ and $u_i < u_{i-1}$, color the site $(1, i)$ in the array black ($\bullet$). Otherwise, if $u_i > u_{i-1}$, color it white ($\circ$). Similarly, if $w = \sigma_S u$ and $i \in T$, color site $(2, i)$ black if $w_i < w_{i-1}$ and white if $w_i > w_{i-1}$. We refer to a site which is not colored as {\it empty}. Any $2\times n$ array which arises in this way will be called a {\it diagram}. Let $C(D)$ denote the set of words $u$ which (together with some sets $S,T$) give $D$ as above. An example of a diagram is given in Figure \ref{fi_di1}, where its associated {\it particle trajectories} are also given. Note that these lines may be added or omitted as we please -- they are determined by the coloring of the $2\times n$ array. Formally, a trajectory is given by the three positions $p_1, p_2, p_3$ of the particle in $u$, $\sigma_S u$, $\sigma_T \sigma_S u$ respectively. We then say that the particle has {\it visited} position $p_1$ in the upper row and position $p_2$ in the lower row (we will use no such notation for $p_3$). Instead of defining an involution on triples $(u, S, T)$, we will define an involution on diagrams. Though diagrams are considered cyclic, we will most often deal with segments of diagrams.

\begin{figure}
	\begin{tikzpicture}
		\node[draw,shape=circle] at (2,2){};
		\node[circle,fill=black] at (3,2){};
		\node[circle,fill=black] at (4,2){};
		\node[circle,fill=black] at (5,2){};
		\node[circle,fill=black] at (3,1){};
		\node[draw,shape=circle] at (4,1){};
		\node[circle,fill=black] at (6,1){};
		\draw(1,2)--(1,0);
		\draw(2,2)--(5,1)--(6,0);
		\draw(3,2)--(2,1)--(3,0);
		\draw(4,2)--(3,1)--(2,0);
		\draw(5,2)--(4,1)--(4,0);
		\draw(6,2)--(6,1)--(5,0);
	\end{tikzpicture}
	\caption{A diagram $D$. The set $C(D)$ is given by all words $u_1\dots u_6$ such that $u_2$ is larger than all other letters $u_1,u_3,u_4,u_5,u_6$, and $u_4 < u_3 < u_5$. For example, the crossing down to the left implies that $u_3 > u_4$.}
	\label{fi_di1}
\end{figure}
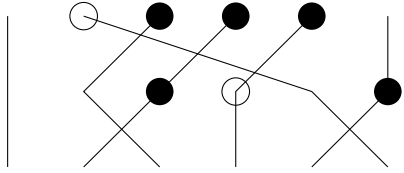

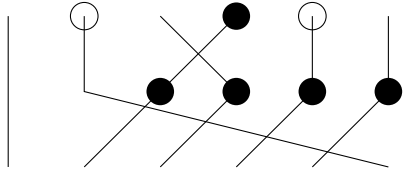
\begin{figure}
	\begin{tikzpicture}
		\node[draw,shape=circle] at (2,2){};
		\node[circle,fill=black] at (4,2){};
		\node[draw,shape=circle] at (5,2){};
		\node[circle,fill=black] at (3,1){};
		\node[circle,fill=black] at (4,1){};
		\node[circle,fill=black] at (5,1){};
		\node[circle,fill=black] at (6,1){};
		\draw(1,2)--(1,0);
		\draw(2,2)--(2,1)--(6,0);
		\draw(3,2)--(4,1)--(3,0);
		\draw(4,2)--(3,1)--(2,0);
		\draw(5,2)--(5,1)--(4,0);
		\draw(6,2)--(6,1)--(5,0);
	\end{tikzpicture}
	\caption{Another diagram.}
	\label{fi_di2}
\end{figure}

Thus we can think of a diagram as describing the trajectories of the particles $u_1, \dots u_n$. 

Say that two diagrams $D$, $D'$ of the same length $n$ are {\it compatible} if the following conditions hold.
\begin{enumerate}
	\item The order relations implied by $D$ and $D'$ are equivalent. That is, $C(D) = C(D')$. 
	\item The number of colored sites in the top row of $D$ equals the number of colored sites in the bottom row of $D'$ and conversely.
	\item The number of black sites passed by particle $i$ is the same in $D$ as in $D'$, for each $i$, and the same holds for white sites.
\end{enumerate}

As an example, the diagrams in Figues \ref{fi_di1} and \ref{fi_di2} are compatible. Not all assignments of $\{$white, black, empty$\}$ to the sites of a $2 \times n$ array gives a diagram - for example, the $2 \times n$ array with only two colored sites, one white atop a black one is not a diagram (if a particle failed to jump the first time and the particle in front of it did not move, it will fail again).

To prove the theorem, it suffices to construct an involution on diagrams pairing up compatible diagrams - condition (3) guarantees that two compatible diagrams have the same weight (for any input word $u \in S(D) = S(D')$).

We will construct the involution $\alpha$ by successively restricting the set of all diagrams to yet smaller classes of diagrams, and show how to 'lift' any definition of an involution on a smaller class to a bigger one. These restrictions all involve finding some subdiagram and replacing it by a smaller one. They are defined as follows.

\begin{enumerate}
\item
	For a diagram $D$ with the following sub-$2\times 2$-diagram in column $c$ and $c+1$,
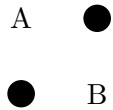
\[
	\begin{tikzpicture}
	\node[fill,shape=circle] at (1,1){};
	\node[fill,shape=circle] at (0,0){};
	\node at (0,1) {A};
	\node at (1,0) {B};
	\end{tikzpicture}
\]
let $\rei^{(c)}(D)$ denote the diagram where the sub-$2\times 2$-diagram is replaced by the following $2\times 1$-diagram (in a single new column replacing $c$, $c+1$).
\[
	\begin{tikzpicture}
	\node at (0,1) {A};
	\node at (0,0) {B};
	\end{tikzpicture}
\]

Here, $A,B$ are placeholders for any of $\bullet$, $\circ$, or an empty space. 
\item
	\begin{enumerate}
		\item	If $D$ contains the $2\times 2$ subdiagram
	
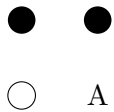
\[
		\begin{tikzpicture}
		\node[fill,shape=circle] at (1,1){};
		\node[fill,shape=circle] at (2,1){};
		\node[draw,shape=circle] at (1,0){};
		\node at (2,0){A};
		\end{tikzpicture}
	\]
	in column $c$ and $c+1$, let $\reiia^{(c)}(D)$ denote the result of replacing the $2\times 2$ diagram by the $2\times 1$ diagram

\[
	\begin{tikzpicture}
	\node[fill,shape=circle] at (1,1){};
	\node at (1,0){A};
	\end{tikzpicture}
\]

	\item
	If $D$ is a diagram containing the $2\times 2$ subdiagram

\[
	\begin{tikzpicture}
	\node at (1,1){A};
	\node[fill,shape=circle] at (1,0){};
	\node[fill,shape=circle] at (2,0){};
	\node[draw,shape=circle] at (2,1){};
	\end{tikzpicture}
\]
	in column $c$ and $c+1$, let $\reiib^{(c)}(D)$ denote the result of replacing the $2\times 2$ diagram by the $2\times 1$ diagram

\[
	\begin{tikzpicture}
	\node[fill,shape=circle] at (1,0){};
	\node at (1,1){A};
	\end{tikzpicture}
\]
	\end{enumerate}

	\item
		If $D$ contains the sub-$2\times 2$-diagram
	\[
		\begin{tikzpicture}
		\node[fill,shape=circle] at (1,1){};
		\node[draw,shape=circle] at (0,0){};
		\node at (0,1) {A};
		\node at (1,0) {B};
		\end{tikzpicture}
	\]
	in column $c$ and $c+1$, let $\reiii^{(c)}(D)$ denote the diagram with the subdiagram replaced by the following.
	\[
		\begin{tikzpicture}
		\node at (0,1) {A};
		\node at (0,0) {B};
		\end{tikzpicture}
	\]

\end{enumerate}

To find $\alpha(D)$ for a diagram $D$, we first perform three kinds of reductions. After these reductions we make a small change to the diagram, and then invert the reductions. The whole process is illustrated in Figure \ref{fi_4}. The proofs that the reductions 'work' are all similar -- we do all the details for one case in Lemma \ref{le_red2} and leave the remainder to the reader.

\newcommand{\bb}[2]{\node[circle,fill=black] at (#1,#2){};}
\newcommand{\wb}[2]{\node[draw,shape=circle] at (#1,#2){};}
\newcommand{\nb}[2]{}

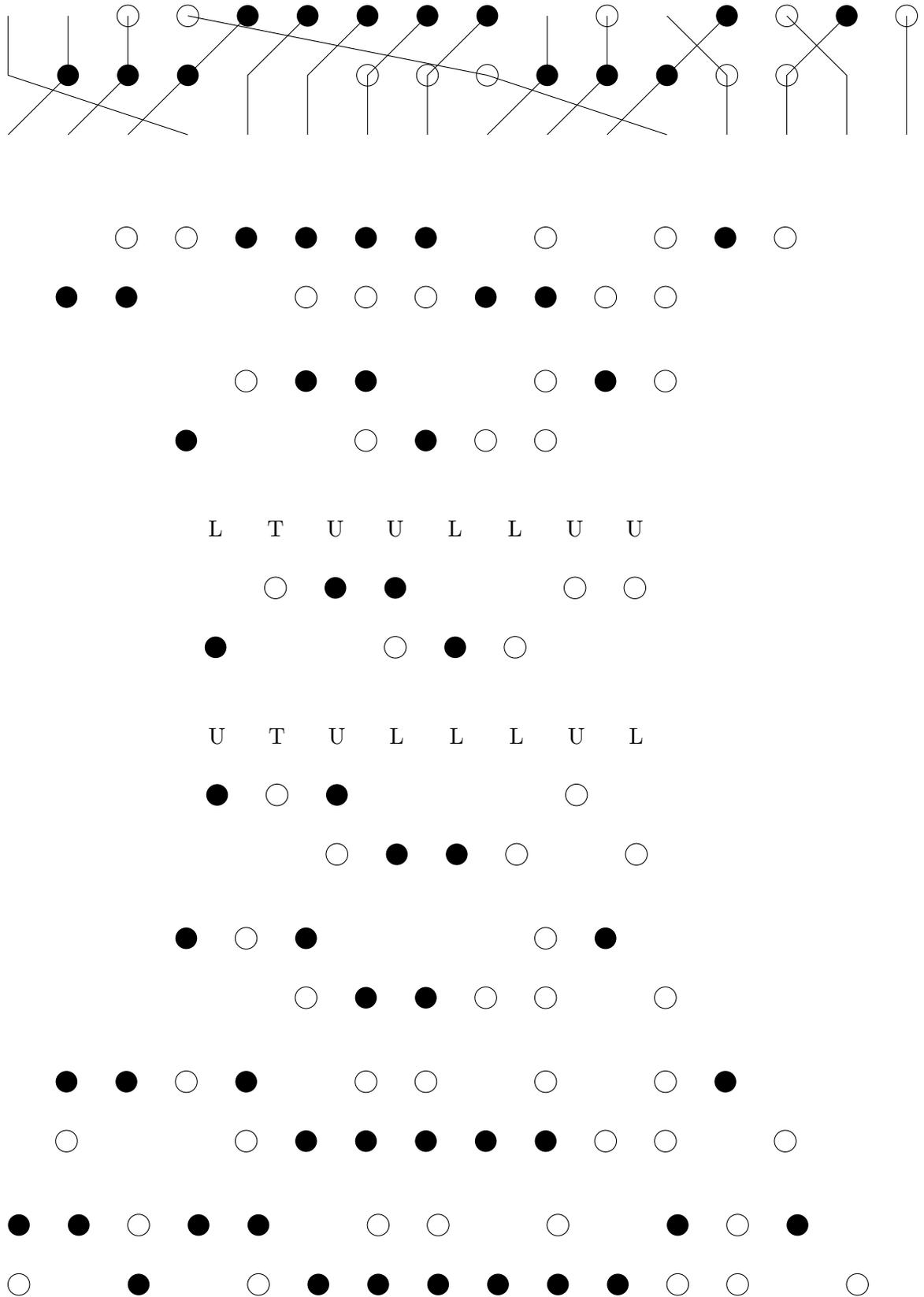
\begin{figure}
	\begin{tikzpicture}
\wb{2}{2};\wb{3}{2};\bb{4}{2};\bb{5}{2};\bb{6}{2};\bb{7}{2};\bb{8}{2};\wb{10}{2};\bb{12}{2};\wb{13}{2};\bb{14}{2};\wb{15}{2};
\bb{1}{1};\bb{2}{1};\bb{3}{1};\wb{6}{1};\wb{7}{1};\wb{8}{1};\bb{9}{1};\bb{10}{1};\bb{11}{1};\wb{12}{1};\wb{13}{1};
\draw(0,2)--(0,1)--(3,0);
\draw(1,2)--(1,1)--(0,0);
\draw(2,2)--(2,1)--(1,0);
\draw(3,2)--(8,1)--(11,0);
\draw(4,2)--(3,1)--(2,0);
\draw(5,2)--(4,1)--(4,0);
\draw(6,2)--(5,1)--(5,0);
\draw(7,2)--(6,1)--(6,0);
\draw(8,2)--(7,1)--(7,0);
\draw(9,2)--(9,1)--(8,0);
\draw(10,2)--(10,1)--(9,0);
\draw(11,2)--(12,1)--(12,0);
\draw(12,2)--(11,1)--(10,0);
\draw(13,2)--(14,1)--(14,0);
\draw(14,2)--(13,1)--(13,0);
\draw(15,2)--(15,1)--(15,0);
	\end{tikzpicture}

	\vspace{1.5cm}

	\begin{tikzpicture}
\wb{2}{2};\wb{3}{2};\bb{4}{2};\bb{5}{2};\bb{6}{2};\bb{7}{2};\wb{9}{2};\wb{11}{2};\bb{12}{2};\wb{13}{2};
\bb{1}{1};\bb{2}{1};\wb{5}{1};\wb{6}{1};\wb{7}{1};\bb{8}{1};\bb{9}{1};\wb{10}{1};\wb{11}{1};
	\end{tikzpicture}

	\vspace{1cm}

	\begin{tikzpicture}
\wb{2}{2};\bb{3}{2};\bb{4}{2};\wb{7}{2};\bb{8}{2};\wb{9}{2};
\bb{1}{1};\wb{4}{1};\bb{5}{1};\wb{6}{1};\wb{7}{1};
	\end{tikzpicture}
	
	\vspace{1cm}

	\begin{tikzpicture}
\wb{2}{2};\bb{3}{2};\bb{4}{2};\wb{7}{2};\wb{8}{2};
\bb{1}{1};\wb{4}{1};\bb{5}{1};\wb{6}{1};
\node at (1,3){L};\node at (2,3){T};\node at (3,3){U};\node at (4,3){U};\node at (5,3){L};\node at (6,3){L};\node at (7,3){U};\node at (8,3){U};
	\end{tikzpicture}
		
\vspace{1cm}

	\begin{tikzpicture}
\bb{1}{2};\wb{2}{2};\bb{3}{2};\wb{7}{2};
\wb{3}{1};\bb{4}{1};\bb{5}{1};\wb{6}{1};\wb{8}{1};
\node at (1,3){U};\node at (2,3){T};\node at (3,3){U};\node at(4,3){L};\node at(5,3){L};\node at(6,3){L};\node at (7,3){U};\node at (8,3){L};
	\end{tikzpicture}

\vspace{1cm}

	\begin{tikzpicture}
\bb{2}{2};\wb{3}{2};\bb{4}{2};\wb{8}{2};\bb{9}{2};
\wb{4}{1};\bb{5}{1};\bb{6}{1};\wb{7}{1};\wb{8}{1};\wb{10}{1};
	\end{tikzpicture}

\vspace{1cm}

	\begin{tikzpicture}
\bb{1}{2};\bb{2}{2};\wb{3}{2};\bb{4}{2};\nb{5}{2};\wb{6}{2};\wb{7}{2};\nb{8}{2};\wb{9}{2};\nb{10}{2};\wb{11}{2};\bb{12}{2};
\wb{1}{1};\nb{2}{1};\nb{3}{1};\wb{4}{1};\bb{5}{1};\bb{6}{1};\bb{7}{1};\bb{8}{1};\bb{9}{1};\wb{10}{1};\wb{11}{1};\nb{12}{1};\wb{13}{1};
	\end{tikzpicture}

\vspace{1cm}

	\begin{tikzpicture}
\bb{1}{2};\bb{2}{2};\wb{3}{2};\bb{4}{2};\bb{5}{2};\wb{7}{2};\wb{8}{2};\nb{9}{2};\wb{10}{2};\nb{11}{2};\bb{12}{2};\wb{13}{2};\bb{14}{2};
\wb{1}{1};\nb{2}{1};\bb{3}{1};\wb{5}{1};\bb{6}{1};\bb{7}{1};\bb{8}{1};\bb{9}{1};\bb{10}{1};\bb{11}{1};\wb{12}{1};\wb{13}{1};\wb{15}{1};
	\end{tikzpicture}

	\caption{Example of computing $\alpha(D)$ for the diagram $D$ in the top. The three diagrams after the first one correspond to I-, II- and III-reductions. The next step consists of computing $\alpha(D')$ for the III-reduced diagram $D'$. The remaining steps consists of inverting the III-, II- and I-reductions. So the last diagram is the image under the involution of the first. To reduce clutter we have drawn the trajectories only for the first diagram.}
	\label{fi_4}
\end{figure}

\newcommand{\vac}{\circ}
\newcommand{\occ}{\bullet}

\begin{lemma}
Let $D, D'$ be two diagrams and $c$ a column for which $\rei^{(c)}$ is applicable.
If $\rei^{(c)}(D)$ and $\rei^{(c)}(D')$ are compatible, then $D$ and $D'$ are compatible.
\end{lemma}

Call a diagram $D$ for which there is no $c$ such that $\rei^{(c)}$ can be applied a {\it I-reduced diagram}. The lemma shows that to construct our required involution it suffices to construct it on the set of I-reduced diagrams. An example of our involution $\alpha$ to be defined on this restricted set is given by removing the forbidden subdiagram from Figures \ref{fi_di1} and \ref{fi_di2}.

It is easy to see that for a I-reduced diagram, two sites in the same column cannot both be black, and if there is a column colored black in the lower row and white in the upper row, then the site in the lower row to the left of the column is black.

\begin{lemma}
\label{le_red2}
Suppose $D$ and $D'$ are I-reduced diagrams, $c$ a column. Then
\begin{itemize}
	\item $D$ and $D'$ are compatible whenever $\reiia^{(c)}(D)$ and $\reiia^{(c)}(D')$ are both defined and compatible.
	\item $D$ and $D'$ are compatible whenever $\reiia^{(c)}(D)$ and $\reiib^{(c)}(D')$ are both defined and compatible.
	\item $D$ and $D'$ are compatible whenever $\reiib^{(c)}(D)$ and $\reiib^{(c)}(D')$ are both defined and compatible.
\end{itemize}
\end{lemma}
\begin{proof}
	We consider the second case. Denote the placeholders in $D$ and $D'$ by $A$ and $A'$. Note that since the diagrams are I-reduced, $A$ (and $A'$) is not colored black.
	First we need to show that the output of $D$ and $D'$ are the same (for any input word $u \in C(D) = C(D')$), that is, we need to show that each particle ends up in the same place in both $D$ and $D'$. We know this is true for $\reiia^{(c)}(D)$ and $\reiib^{(c)}(D')$ so we only need to consider particles that pass through column $c$ in $D$. This is easy to check -- if the particle passes column $c$ in $\reiia^{(c)}(D)$ it must do so in the upper row, and then it must pass column $c$ in the lower row in $\reiib^{(c)}(D')$. The column added when constructing $D$ respectively $D'$ preserves this property. And if the particle does not pass column $c$ in $\reiia^{(c)}(D)$ and $\reiib^{(c)}(D')$, the same holds in $D$ and $D'$. Finally it's easy to check that the particle starting in the 'new' column in $D$ and $D'$ end up in the same place (indeed, it's the only place left).
	Next, we need to show that $C(D) = C(D')$. This amounts to checking what new relations on the input word $u$ the colors of the new sites in $D$ and $D'$ give us. This can be checked directly and indivdually for the new white site in $D$ and $D'$ and the new black site in $D$ and $D'$.
\end{proof}

Call a I-reduced diagram for which no $\reiib^{(c)}$ nor $\reiia^{(c)}$ is defined a II-reduced diagram. We now make the final restriction.

\begin{lemma}
Suppose $D, D'$ are II-reduced diagrams, and $c$ is a column for which $\reiii^{(c)}(D)$ and $\reiii^{(c)}(D')$ are both defined. Then, if the latter two are compatible, then so are $D$ and $D'$.
\end{lemma}

If $D$ is a II-reduced diagram for which no $\reiii^{(c)}$ is defined we call $D$ a III-reduced diagram. The variety of III-reduced diagrams is sufficiently narrow to be analyzed directly. Note that in any III-reduced diagram, if a particle passes two colored sites, then those are necessarily white. Suppose $D$ is a III-reduced diagram. Above the starting column of each particle (see Figure \ref{fi_4}), write either $U, L$ or $T$, depending on whether the particle passes (T)wo sites colored white or a unique colored site in the (U)pper row or a unique colored site in the (L)ower row. To define $\alpha(D)$, for each maximal word $U^r D^s$ strictly between two words of the type $T$ or $LU$, change the behavior of the particles corresponding to $U^rL^s$ so that it will read $U^s L^r$ instead (see the example in Figure \ref{fi_4}). We define the resulting diagram to be $\alpha(D)$.

It is readily checked that the map is an involution, and that it pairs up compatible diagrams. Thus this is true of the extended involution on all diagrams, too. This finishes the proof.

\section{Questions, remarks}
\label{sec_que}

\subsection{}
 Is there an easily described conjugation matrix $U_{i,J}$ for general links $(i, J)$ in type C (or B or D)?

As an example, in the notation of Section \ref{sec_C}, consider the link $(i, J, J') = (2, \{3,4\}, \{2,3,4\})$ when $n=4$. Convert states in $\Theta_J$ (and $\Theta_{J'}$) to words $w_1w_2w_3w_4$. If the $i$th column is empty, let $w_1 = 3$ ($2$), otherwise let $w_i = +i$ if there is a particle of class $i$ in the upper line, and $w_i = -i$ if there is a particle of class $i$ in the lower line. We order the $48$ states of $\Theta_J$ lexicographically: $\bar{2}\bar{1}33 < \bar{2}133 < \bar{2}3\bar{1}3 < \dots < 3321$. Similarly, the $8$ states of $\Theta_{J'}$ are ordered (from left to right) as $\bar{1}222 < 1222 < 2\bar{1}22 < 2122 < 22\bar{1}2 < 2212 < 222\bar{1} < 2221$. Indexing rows and columns this way, the transpose $U$ of the following matrix satisfies $M_J U = U M_{J'}$. 

$
\left(
\begin{array}{c}
1 0 1 0 2 1 0 3 0 1 0 0 0 4 0 2 0 0 2 0 2 0 4 2 1 0 2 1 0 2 0 0 0 1 0 0 1 0 2 1 4 2 0 1 0 0 2 1 \\
1 0 1 0 1 0 0 3 0 1 1 1 0 4 0 2 2 2 2 0 2 0 2 0 1 0 1 0 0 2 1 1 0 1 1 1 1 0 1 0 3 1 1 2 1 1 1 0 \\
1 0 1 1 2 1 0 3 0 0 0 0 0 4 0 0 0 0 2 0 2 2 4 2 1 1 2 1 0 1 0 0 0 0 0 0 1 1 2 1 4 3 0 0 2 1 0 0 \\
1 0 1 0 0 0 0 3 0 1 2 1 0 4 0 2 4 2 2 0 2 0 0 0 1 0 0 0 0 2 2 1 0 1 2 1 1 0 0 0 2 2 2 1 2 1 0 0 \\
1 1 1 1 2 1 0 2 0 0 0 0 0 2 0 0 0 0 2 2 2 2 4 2 1 2 2 1 0 0 0 0 1 1 2 1 0 0 0 0 4 3 0 0 2 1 0 0 \\
1 0 0 0 0 0 0 3 1 1 2 1 0 4 2 2 4 2 2 0 0 0 0 0 0 1 0 0 1 1 2 1 1 1 2 1 0 0 0 0 4 3 0 0 2 1 0 0 \\
1 2 1 1 2 1 0 1 0 0 0 0 1 3 1 1 2 1 1 1 1 1 2 1 1 2 2 1 0 0 0 0 1 1 2 1 0 0 0 0 4 3 0 0 2 1 0 0 \\
0 1 0 0 0 0 1 2 1 1 2 1 2 4 2 2 4 2 0 0 0 0 0 0 1 2 2 1 0 0 0 0 1 1 2 1 0 0 0 0 4 3 0 0 2 1 0 0 \\
\end{array}
\right)
$

(This is a $8\times 48$ matrix all of whose entries are in $\{0,1,2,3,4\}$.)
Is there a combinatorial rule (along the lines of the queueing process of Section \ref{sec_fi}) which produces the column corresponding to any given state in $\Theta_{J'}$?

\subsection{} Is it possible to carry the explicit description of the stationary distribution $\Theta_J$ further than is done in Section \ref{sec_se}, say for $|J| = n - 2$?

\subsection{} Can the matrix $U$ in Section \ref{sec_fi} be defined without reference to the permutation representation of the group, ie. using only the realization of the group as a reflection group?

\subsection{} Small examples indicate that a similar queuing process exists for groups of type B and D. Is it easier to extend the analysis in Sections \ref{sec_fi} and \ref{sec_se} for these groups?

\subsection{} Can the $k$-TASEP be extended to general Weyl groups?

\end{document}